\documentclass[a4paper]{article}
\usepackage[margin=1 in]{geometry}

\usepackage{verbatim}
\usepackage{hyperref}
\usepackage{amsmath}
\usepackage{enumerate, amsthm}
\usepackage{mathrsfs}
\usepackage{xcolor}
\usepackage{mathtools}
\usepackage{tikz}
\usetikzlibrary{patterns}
\usetikzlibrary{arrows}
\usepackage{ifthen}
\usepackage{bm}
\usepackage{amssymb}
\usepackage{centernot}

\numberwithin{equation}{section}

\theoremstyle{plain}
\newtheorem{theorem}{Theorem}[section]
\newtheorem{lemma}[theorem]{Lemma}
\newtheorem{corollary}[theorem]{Corollary}
\newtheorem{proposition}[theorem]{Proposition}

\theoremstyle{definition}

\theoremstyle{remark}
\newtheorem{remark}[theorem]{Remark}

\renewcommand{\P}{\mathbb P}
\newcommand{\E}{\mathbb E}
\newcommand{\R}{\mathbb R}
\newcommand{\Z}{\mathbb Z}
\newcommand{\N}{\mathbb N}
\renewcommand{\S}{\mathbb S}

\title{Uniqueness of unbounded component for level sets of smooth Gaussian fields}
\author{Franco Severo\footnotemark[1]\footnote{ETH Z\"{urich}, franco.severo@math.ethz.ch}}

\date{}

\begin{document}
\thispagestyle{empty}
\maketitle

\begin{abstract}
For a large family of stationary continuous Gaussian fields $f$ on $\R^d$, including the Bargmann--Fock and Cauchy fields,
we prove that there exists at most one unbounded connected component in the level set $\{f=\ell\}$ (as well as in the excursion set $\{f\geq\ell\}$) almost surely  for every level $\ell\in \R$, thus proving a conjecture proposed by Duminil-Copin, Rivera, Rodriguez \& Vanneuville.
As the fields considered are typically very rigid (e.g.~analytic almost surely), there is no sort of finite energy property available and the classical approaches to prove uniqueness become difficult to implement. We bypass this difficulty using a soft shift argument based on the Cameron--Martin theorem.
\end{abstract}

\section{Introduction}\label{sec:intro}

Let $f$ be a centered, stationary, ergodic and continuous Gaussian field on $\R^d$ and denote its covariance kernel by
$$\kappa(x)\coloneqq\E[f(0)f(x)],~~~ x\in\R^d.$$
We are interested in the geometry of its \emph{level sets}
$$\{f=\ell\} \coloneqq \{ x \in \R^d :~ f(x) = \ell \},~~~ \ell\in \R,$$
and particularly the \emph{nodal set} $\{f=0\}$.
Two of the most important examples of such Gaussian fields come respectively from the study of ``typical'' algebraic varieties and nodal sets of Laplace eigenfunctions.
The first example is the \emph{Bargmann--Fock field}, which characterized by $\kappa(x)=e^{-\frac12 ||x||_2^2}$. This field arises as the scaling limit of random homogeneous polynomials (the so called Kostlan ensamble) as the degree tends to infinity -- see e.g.~\cite{Bel22}. Therefore, the behaviour of the nodal set of the Bargmann--Fock field is related to that of typical algebraic varieties.
The second example is the \emph{monochromatic random wave}, which is characterized by $\kappa(x)=\int_{\S^{d-1}} e^{i\langle x,\omega \rangle}d\sigma(\omega)$. Similarly to the Bargmann--Fock field, it arises as the scaling limit of random spherical harmonics as the frequency tends to infinity. 

In dimension one, the level set $\{f=\ell\}$ is simply a discrete set of points and the study of its geometry boils down to the distribution of these point in space. A first answer to this problem dates back to the works of Kac \cite{Kac43} and Rice \cite{Ric44} in the 40's, where they computed the expected number of such points in a given region. In higher dimensions though, level sets are smooth hypersurfaces of codimension one, and studying their geometry becomes a much more challenging task. Even the seemingly simple problem of computing the (expected) number of connected components (to which we henceforth simply refer to as components for short) resisted for decades before being solved in the breakthrough work Nazarov \& Sodin \cite{NaSo09}. After that, other topological observables such as Betti numbers have been considered \cite{GW14,GW16}, culminating in an asymptotic law for the number of components with given topological types \cite{SaWi19,CaSa19}.

Despite the great progress described above, all the aforementioned works have the downside of failing to distinguish bounded and unbounded components, treating both types equally. As a complementary line of research, the large scale connectivities of level sets of continuous Gaussian fields and their connections to percolation theory have been the object of many recent works.
In this direction, it is more convenient to consider the \emph{excursion sets}
$$\{f\geq\ell\} \coloneqq \{ x \in \R^d :~ f(x) \geq \ell \},~~~ \ell\in \R.$$
By monotonicity in $\ell$, it makes sense to define the \emph{percolation threshold}
$$\ell_c\coloneqq\sup\{\ell\in \R:\, \P\big[\text{there is an unbounded component in } \{f\geq\ell\}\big]>0\}.$$
Since the level set $\{f=\ell\}$ is simply the boundary of the excursion set $\{f\geq \ell\}$, one can use topological arguments to deduce properties of the former from those of the latter. For instance, it turns out that the existence of an unbounded component for $\{f=\ell\}$ is equivalent to the existence of unbounded components for both $\{f\geq\ell\}$ and $\{f\leq\ell\}$ simultaneously. By the symmetry of $f$ around $0$, the distribution of $\{f\leq\ell\}$ is equal to that of $\{f\geq-\ell\}$, and one readily deduces that $\{f=\ell\}$ contains an unbounded component for $\ell\in (-\ell_c,\ell_c)$, while for $\ell\notin[-\ell_c,\ell_c]$ all components of $\{f=\ell\}$ are bounded.

Most of the progress made in this field concerns the case $d=2$. The main reason for this is a very special \emph{duality property} which is only available in the planar case. In particular, this property implies that $\ell_c=0$ for every planar Gaussian fields satisfying very mild assumptions \cite{MRV20}. Many works have been devoted to proving the so called \emph{sharpness of phase transition}, i.e.~the fact that the bounded components are ``tiny'' (or microscopic) for every non-critical level $\ell\neq\ell_c=0$ -- see e.g.~\cite{RVb,MV,MRV20} for sharpness results in 2D and \cite[Theorem 1.3]{MS22} for the precise notion of ``tiny''. In contrast, at the critical level $\ell_c=0$, which corresponds to the nodal set, it is expected that all components are still bounded, but some are ``large'' (or macroscopic), which has only been proved for positively correlated fields -- see \cite{Ale96,BG17}. Beyond that, it has been conjectured by Bogomolny \& Schmit \cite{BogSch02} that, for the monochromatic random wave in 2D, these nodal lines converge to SLE(6) after a scaling limit, exactly as for critical planar percolation \cite{Smi01}. This is also expected to hold for the Bargmann--Fock field and other planar fields with sufficiently fast decay of correlations.

The case $d\geq3$ is much less understood. Since one has $\ell_c=0$ on the plane, it is natural to expect that it is ``strictly easier'' to have an unbounded component in higher dimensions, i.e.~that $\ell_c>0$ as long as $d\geq3$. This was conjectured by Sarnak \cite{Sa1} for the monochromatic random wave, but is expected to hold for basically any continuous Gaussian fields. Very recently, Duminil-Copin, Rivera, Rodrigues \& Vanneuville \cite{DRRV21} proved that indeed $\ell_c>0$ in dimensions $d\geq3$ for all fields with positive and sufficiently fast decay of correlations, which includes in particular the Bargmann--Fock field. Rivera \cite{Riv21} then used this result and an approximation argument to prove that the same holds for the monochromatic random wave in sufficiently large dimensions. Despite proving the existence of an unbounded nodal component for $\{f=0\}$ in dimensions $d\geq3$, \cite{DRRV21} left open the question of uniqueness. This is precisely what we address in the present article. In order to state our results, we will impose the following assumptions on the field $f$.

	\begin{enumerate}[I]
		\item \label{a:1} \emph{(Smoothness)}  $\kappa\in C^{8}$.
		\item \label{a:2} \emph{(Non-degeneracy)} The Gaussian vector
		$(f(0),\nabla f(0))\in \R^{d+1}$
		is non-degenerate.
		\item \label{a:3} \emph{(Tame spectrum)} The Fourier transform $\hat{\kappa}$ (a positive measure, by Bochner's theorem) has uniformly positive density in a neighbourhood of $0$.
	\end{enumerate}
We can now state our main result, which proves Conjecture 1.8 of \cite{DRRV21}. 

\begin{theorem}\label{thm:uniq} If $f$ satisfies assumptions \ref{a:1}, \ref{a:2} and \ref{a:3}, then for every $\ell\in\R$ there exist at most one unbounded connected component in $\{f\geq\ell \}$ almost surely.
\end{theorem}

By standard topological arguments, one can deduce the following.

\begin{corollary}\label{cor:uniq_nodal}
	If $f$ satisfies assumptions \ref{a:1}, \ref{a:2} and \ref{a:3}, then for every $\ell\in \R$ there exist at most one unbounded connected component in $\{f=\ell\}$ almost surely.
\end{corollary}

\begin{remark}\label{rem:assumptions}
	Assumption~\ref{a:3} is automatically satisfied if $\kappa\in L^1$ (which implies $\hat{\kappa}\in C^0$) and $\hat{\kappa}(0)=\int_{\R^d}\kappa(x)dx>0$ -- notice that this does not require positive correlations, i.e.~$\kappa\geq0$. In particular, our assumptions are satisfied by the Bargmann--Fock field. It is straightforward to check that the Cauchy fields (given by $\kappa(x)=(1+||x||_2^2)^{-\alpha/2}$) also satisfy our assumptions for every exponent $\alpha>0$. We stress that this includes the strongly correlated cases, i.e.~$\alpha\in(0,d]$.
	Finally, we note that  assumption \ref{a:3} is not satisfied by the monochromatic random wave, for which uniqueness remains open.
\end{remark}

The proof of Theorem~\ref{thm:uniq} follows the approach of Burton \& Keane \cite{BurKea89}. The proof of \cite{BurKea89} strongly relies on a \emph{finite energy} property -- or more precisely, \emph{insertion tolerance}, see e.g.~\cite[Theorem 7.8]{LP16}. Roughly speaking, this property says that conditionally on what happens in the complement of a bounded region $U$, the probability of a desired event in $U$ remains positive. However, the fields we consider here are often analytic functions almost surely (which is the case of the Bargmann--Fock field for instance), and in particular its restriction to any open set determines the field globally. Therefore, this crucial property is not available in our context. A naive approach would be to apply the Burton--Keane argument to a discretized version $f_\varepsilon$ of the field $f$ -- say, its restriction to the lattice $\varepsilon\Z^d$ -- and then take the limit $\varepsilon\to 0$. This strategy runs into two problem. First, it is not clear how to prove that $f_\varepsilon$ has finite energy. Second (and more importantly!), even if $f_\varepsilon$ does have finite energy, it is not clear how to properly take the limit $\varepsilon\to0$. Indeed, uniqueness is a highly non-local event and the Burton--Keane argument, being based on ergodicity, gives no quantitative information in finite volume, thus making it difficult to perform a limiting argument.

We overcome the lack of finite energy by using a soft shift argument based on the Cameron--Martin theorem for Gaussian fields. 
This technique has been recently used by the author \cite{Sev21} to prove sharpness of phase transition in all dimensions for positively and weakly correlated fields. A crucial difference with \cite{Sev21} is that in our case we need to further ensure that, despite the unbounded support of the shift function, the effective modification is only local, which is ultimately guaranteed by assumption \ref{a:3} -- see Lemmas~\ref{lem:global_equiv} and \ref{lem:tame_shift}.

\section{Preliminaries}\label{sec:prelim}

In this section we overview a few basic properties of Gaussian fields. In what follows, $B_R(x)$ denotes the open Euclidean ball of radius $R$ centered at $x$. The closure and boundary of $B_R(x)$ are denoted by $\overline{B_R}(x)$ and $\partial B_R(x)$, respectively. We may omit $x$ from the notation when $x=0$. 

The following basic lemma provides an upper bound on the expected number of components intersecting the sphere by its area. The result is well known and the proof follows standard arguments based on the Kac--Rice formula. We choose to include the precise formulation we need and its proof here for the sake of completeness. Obtaining the precise asymptotics for the (expected) number of components is a much more challenging problem -- see \cite{NaSo09}.

\begin{lemma}\label{lem:number_comp}
	Assume that $f$ satisfies assumptions \ref{a:1} and \ref{a:2}. Then there exists a constant $C_0\in(0,\infty)$ such that for every $L\geq1$ and $\ell\in \R$, one has
	$$\E[\mathcal{N}_{\partial B_L}(f,\ell)]\leq C_0L^{d-1},$$
	where $\mathcal{N}_{\partial B_L}(f,\ell)$ denotes the number of components in $\{ f\geq\ell \}\cap \partial B_L$.
\end{lemma}

\begin{proof}
	By compactness and continuity, every component of $\{ f\geq\ell \}\cap \partial B_L$ contains a local maximum of $f$ on $\partial B_L$, thus a critical point. Hence $\mathcal{N}_{\partial B_L}\leq N\coloneqq |\{x\in \partial B_L:~\nabla f(x)=0\}|$, where $\nabla$ denotes the spherical gradient. 
	By assumptions~\ref{a:1} and \ref{a:2}, we can apply of the Kac--Rice formula to the field $\nabla f $ -- see Theorem~6.2 and Proposition~6.5 of \cite{AW09} -- to obtain 
	$$\E[N]=\int_{\partial B_L} \E[|\det \nabla^2 f(x)| \,\big|\, \nabla f(x)=0] p_{\nabla f(x)}(0) dx,$$
	where $p_{\nabla f(x)}(0)$ denotes the density function of $\nabla f(x)$ at $0$. Since $f$ is stationary and the curvature of $\partial B_L$ is bounded for $L\geq1$, there exists $C\in(0,\infty)$ such that $\E[|\det \nabla^2 f(x)| \,\big|\, \nabla f(x)=0] p_{\nabla f(x)}(0) \leq C$ for every $x$, and the proof is complete.
\end{proof}

Given two subsets $A\subset B\subset \R^d$, we say that $A$ and $B$ are \emph{percolation equivalent} if the natural inclusion map from the components of $A$ onto the components of $B$ is a bijection that furthermore maps bounded components onto bounded components.
Given a function $h:\R^d\to [0,\infty)$, let $\textrm{GE}(f,\ell,h)$ be the event that there exists $R>0$ such that $\{f\geq\ell\}\setminus B_R$ and $\{f+h\geq\ell\}\setminus B_R$ are percolation equivalent. We refer to this event as ``global equivalence''. Intuitively, global equivalence means that, from a percolation point of view, shifting the field $f$ by $h$ only changes things locally. 

The following lemma states that global equivalence holds almost surely whenever the shift function $h$ and its derivatives decay sufficiently fast. This lemma will be a crucial ingredient in the proof of Theorem~\ref{thm:uniq}.

\begin{lemma}\label{lem:global_equiv}
	Assume that $f$ satisfies assumptions \ref{a:1} and \ref{a:2}. Then for every
	non-negative function $h\in C^2\cap W^{2,\infty}\cap W^{1,1}$, and every $\ell\in\R$,
	$$\P[\emph{\textrm{GE}}(f,\ell,h)]=1.$$
\end{lemma}

\begin{proof}
	We will work with the annuli $A(R,M)=\overline{B_M}\setminus B_R$, $M>R>0$, and its natural decomposition $A(R,M)=A^0(R,M)\cup A^1(R,M)$ with $A^0(R,M)\coloneqq B_M\setminus \overline{B_R}$ and $A^1(R,M)\coloneqq \partial B_R\cup\partial B_M$. In what follows, $\nabla$ denotes either the standard Euclidean gradient on $A^0(R,M)$ or the spherical gradient on $A^1(R,M)$.  Consider the sets $\mathcal{E}_t\coloneqq\{f+th\geq\ell\}$, $t\in[0,1]$, interpolating continuously between $\mathcal{E}_0=\{f\geq\ell\}$ and $\mathcal{E}_1=\{f+h\geq\ell\}$.
	By definition, if $\textrm{GE}(f,\ell,h)$ does not happen then for every $R\geq 1$ there exists $M_0=M_0(R)>R$ such that for every $M\geq M_0$ at least one of the following happens:
	\begin{itemize}
		\item (Merging) there exist two distinct components of $\mathcal{E}_0\cap A(R,M)$ that get connected in $\mathcal{E}_1\cap A(R,M)$,
		\item (Emergence) $\mathcal{E}_1\cap A(R,M)$ contains a component disjoint from those of $\mathcal{E}_0\cap A(R,M)$,
		\item (Explosion) $\mathcal{E}_0\cap A(R,M)$ contains a component that does not intersect $\partial B_{M_0}$, but that does intersect $\partial B_M$ in $\mathcal{E}_1\cap A(R,M)$.
	\end{itemize} 
	Fix such a pair $R,M$. By considering the ``first time'' $t$ at which merging, emergence or explosion of components occurs and its ``location'' $x\in A(R,M)$, one finds $t$ and $x$ such that $x$ is a critical point of $f+th$ with value $\ell$, i.e.~$\nabla f(x)+t\nabla h(x)=0$ and $f(x)+th(x)=\ell$ -- emergence and explosion correspond to local maxima (either in $A^0(R,M)$ or on $A^1(R,M)$ in the case of emergence, but necessarily on $\partial B_M$ in the case of explosion) and merging corresponds to saddle points (again, either in $A^0(R,M)$ or on $A^1(R,M)$). As a conclusion, we have $N^h_0(R,M)+N^h_1(R,M)\geq1$, where $N^h_i(R,M)\coloneqq |\{(x,t)\in A^i(R,M)\times[0,1]:\, \nabla f(x)+t\nabla h(x)=0 \text{ and } f(x)+th(x)=\ell\}|$, $i\in\{0,1\}$. All in all, we conclude that
	\begin{equation}\label{eq:GE_bound}
		\P[\textrm{GE}(f,\ell,h)]\geq \limsup_{R\to\infty}\limsup_{M\to\infty} \P[N_0^h(R,M)+N_1^h(R,M)=0].
	\end{equation}
	
	Notice that $N_i^h(R,M)$ simply counts the number of zeros of the Gaussian fields $g_i(x,t)\coloneqq (f(x)+th(x)-\ell, \nabla f(x)+t\nabla h(x))$ on $A^i(R,M)\times[0,1]$, $i\in\{0,1\}$. We can then apply the Kac--Rice formula to $g_0$ -- see again Theorem~6.2 and Proposition~6.5 of \cite{AW09} -- to obtain
	$$\E[N_0^h(R,M)]=\int_{0}^{1}\int_{A^0} \E[|\det g'_0(x,t)| \,\big|\, g_0(x,t)=0] p_{g_0(x,t)}(0) dxdt,$$
	where $g'_0$ is the Jacobian matrix of $g_0$ and $p_{g_0(x,t)}(0)$ is the density function of $g_0(x,t)$ evaluated at $0$. Now notice that $|\det g'_0|\leq C_d |\partial_t g_0| \prod_{i=1}^{d}|\partial_{x_i} g_0|$ and $|\partial_t g_0(x,t)|= \max\{|f(x),|\nabla f(x)|\}\leq |f(x)|+|\nabla f(x)|$, where $|\cdot|$ denotes the $\ell^\infty$ norm on $\R^d$ or $\R^{d+1}$. Also, by stationarity and boundedness of $h$, $\nabla h$ and $\nabla^2 h$ (since $h\in W^{2,\infty}$), one can easily see that there exists a constant $C=C(f,h)$ such that $\E[\prod_{i=1}^{d}|\partial_{x_i} g_0(x,t)| \,\big|\, g_0(x,t)=0]p_{g_0(x,t)}(0)\leq C$ for all $x,t$. We then conclude that $\E[N_0^h(R,M)]\leq C\int_{A^0(R,M)} |h(x)|+|\nabla h(x)| dx$. Analogously, we have $\E[N_1^h(R,M)]\leq C\int_{A^1(R,M)} |h(x)|+|\nabla h(x)| dx$.
	Hence, by the Markov inequality we obtain
	\begin{align*}
		\P[N_0^h(R,M)+N_1^h(R,M)=0] \geq 1- C\Big(\int_{A^0(R,M)} |h(x)|+|\nabla h(x)| dx + \int_{A^1(R,M)} |h(x)|+|\nabla h(x)| dx\Big).
	\end{align*}
	\sloppy By the assumption that $\int_{\R^d} h(x)+|\nabla h(x)|dx <\infty$ (since $h\in W^{1,1}$) and Fubini's theorem, we have $\limsup_{R\to\infty}\limsup_{M\to\infty} \P[N_0^h(R,M)+N_1^h(R,M)=0]=1$, which concludes the proof by \eqref{eq:GE_bound}.
\end{proof}

\begin{remark}
	We believe that Lemma~\ref{lem:global_equiv} should also hold if the notion of percolation equivalence is replaced by that of \emph{topological equivalence}, namely the existence of a homeomorphism $\phi:\R^d\to \R^d$ with $\phi(A)=B$, but we will not need this.
\end{remark}

Given a stationary, continuous Gaussian field $f$ on $\R^d$ with covariance kernel $\kappa$, one can define the associated \emph{Cameron--Martin space} $H$. It is defined as the Hilbert space obtained by completing the linear span of $(\kappa(\cdot-x))_{x\in\R^d}$ with respect to the inner product given by 
$$\Big\langle \sum_{x\in X} a_{x} \kappa(\cdot-x),\, \sum_{y\in Y} b_{y} \kappa(\cdot-y)  \Big\rangle_H\coloneqq \sum_{x\in X,\, y\in Y} a_x b_y \kappa(x-y).$$ 
Intuitively, $H$ is the Hilbert space such that $f$ is the standard Gaussian random variable on $H$  -- see e.g.~\cite{janson_1997} for more details. 
The Cameron--Martin theorem states that for every $h\in H$, the distributions of $f$ and $f+h$ are mutually absolutely continuous with an explicit Radon--Nikodym derivative. We will only use the following direct consequence of the Cameron--Martin theorem.
	\begin{lemma}\label{lem:Cameron-Martin}
		Let $H$ be the Cameron--Martin space of $f$. Then, for every $h\in H$ and every measurable set $E\subset \R^{\R^d}$, one has that $\P[f\in E]>0$ if and only if $\P[f+h\in E]>0$.
	\end{lemma}
The next lemma allows us to produce a non-negative shift function $h\in H$ satisfying the assumption of Lemma~\ref{lem:global_equiv} above.
\begin{lemma}\label{lem:tame_shift}
	If assumption~\ref{a:3} is satisfied, then there exists a function $h\in H$ such that $h\geq0$, $h(0)>0$ and $h\in \mathcal{S}$, where $\mathcal{S}$ denotes the Schwartz space (notice that $\mathcal{S}\subset  C^2\cap W^{2,\infty}\cap W^{1,1}$).
\end{lemma}
\begin{proof}Assume that $\hat{\kappa}(x)\geq \varepsilon$ for all $x\in B_\delta$, with $\varepsilon,\delta>0$. It is straightforward to check that $H$ contains every function of the form $h=\kappa\ast \rho$, with $\rho:\R^d\to \R$ satisfying $\int \int \rho(x)\rho(y)\kappa(x-y)dx dy<\infty$. Let $h\in \mathcal{S}$ be such that $h\geq 0$, $h(0)>0$ and $\mathrm{supp}(\mathcal{F} h)\subset B_\delta$, where $\mathcal{F}$ denotes the Fourier operator -- one can take for instance $h\coloneqq (\mathcal{F}^{-1}g)^2$, where $g$ is any non-trivial, non-negative, symmetric bump function supported on $B_{\delta/2}$. Now simply notice that by setting $\rho\coloneqq \mathcal{F}^{-1} \big( \frac{\mathcal{F} h}{\mathcal{F} \kappa} \big)$, one obtains $h=\kappa\ast\rho$ and $\int \int \rho(x)\rho(y)\kappa(x-y)dx dy =\langle \kappa\ast\rho, \rho \rangle = \langle \mathcal{F}{\kappa}\cdot\mathcal{F}{\rho}, \mathcal{F}{\rho} \rangle = \int_{B_\delta} (\mathcal{F} h)^2/\mathcal{F}\kappa \leq ||\mathcal{F}h||_2 /\varepsilon <\infty$.
\end{proof}

\section{Proof of uniqueness}\label{sec:proof}

In this section we prove Theorem~\ref{thm:uniq} and Corollary~\ref{cor:uniq_nodal}. We follow the approach of Burton \& Keane \cite{BurKea89}, which goes as follows. Let $\mathcal{N}_\infty(f,\ell)$ be the number of unbounded components of $\{f\geq\ell\}$. By ergodicity $\P[\mathcal{N}_\infty(f,\ell)=k]\in \{0,1\}$ for all $k\in \N_\infty\coloneqq\{0,1,2,\cdots\}\cup\{\infty\}$. Therefore, there exists a unique $k_0 \in \N_{\infty}$ such that $\mathcal{N}_\infty(f,\ell)=k_0$ almost surely. On the one hand, if we assume that $k_0\in\{2,3,\cdots\}$, then one should be able to merge all components \emph{locally}, thus constructing a unique unbounded component, which leads to a contradiction. On the other hand, if we assume that $k_0=\infty$ (actually $k_0\geq3$ is enough), then by another \emph{local} merging, one would prove that certain ``trifurcations points'' exist with positive probability (and therefore have positive density by stationarity). However, a combinatorial argument of Burton \& Keane \cite{BurKea89} shows that the number of trifurcations is smaller than the number of boundary components, which in turn has zero density by Lemma~\ref{lem:number_comp}, thus leading to a contradiction. 

As no finite energy (or insertion-tolerance) property is available in our setting, we implement the local merging by shifting the field $f$ by an appropriate non-negative function $h\in H$. However, since the shift function $h$ has possibly (in fact, typically) unbounded support, it is not a priori clear whether this modification is only local. This will be guaranteed by Lemmas~\ref{lem:global_equiv} and \ref{lem:tame_shift}.

Theorem~\ref{thm:uniq} follows readily from the following two propositions.

\begin{proposition}\label{prop:no_finite>1}
If $f$ satisfies assumptions \ref{a:1}, \ref{a:2} and \ref{a:3}, then for every $\ell\in\R$ one has
$$\P[\mathcal{N}_\infty(f,\ell)\in \{2,3,\cdots\}]=0.$$
\end{proposition}

\begin{proof}
	Assume by contradiction that $\mathcal{N}_\infty(f,\ell)=k_0$ almost surely with $k_0\in \{2,3,\cdots\}$. Let $\mathrm{Int}(f,\ell,R)$ be the event that all the $k_0$ unbounded components of $\{ f\geq\ell \}$ intersect $B_R$. Since $k_0$ is finite, for a sufficiently large radius $R>0$ one has
	$\P[\mathrm{Int}(f, \ell, R)]\geq 1/2$.
	We will now shift the field $f$ by an appropriate function $h\in H$ in such a way that $\{f+h\geq\ell\}$ has a unique unbounded component, thus obtaining a contradiction by Lemma~\ref{lem:Cameron-Martin}. First, there exists $M=M(R)>0$ large enough such that $\P[\inf_{x\in B_R} f(x)\geq -M]\geq3/4$, hence
	\begin{equation}\label{eq:quasi_uniq>0}
		\P[\mathrm{Int}(f,\ell,R)\cap\{\inf_{x\in B_R} f(x)\geq -M\}]\geq 1/4.
	\end{equation}
	Let $h_0\in H$ be be the function given by Lemma~\ref{lem:tame_shift}. Then there exist $c_0, r_0>0$ such that $h_0(y)\geq c_0$ for all $||y||_\infty\leq r_0$.
	Consider the function $h\in H$ given by
	\begin{equation}\label{eq:h}
		h(x)=\frac{M+\ell}{c_0}\sum_{z\in 2r_0\Z^d\cap B_{R+\sqrt{d}r_0}} h_0(x-z).
	\end{equation}
	One can easily check that by construction $h\geq0$, $h\in C^2\cap W^{2,\infty}\cap W^{1,1}$ and $h(x)\geq M+\ell$ for all $x\in B_R$.
	By construction and the definition of global equivalence we have $\mathrm{Int}(f,\ell,R)\cap\{\inf_{x\in B_R} f(x)\geq-M\}\cap \textrm{GE}(f,\ell,h)\subset \{\mathcal{N}_\infty(f+h,\ell)=1\}$, hence $\P[\mathcal{N}_\infty(f+h,\ell)=1]>0$ by \eqref{eq:quasi_uniq>0} and Lemma~\ref{lem:global_equiv}, and thus $\P[\mathcal{N}_\infty(f,\ell)=1]>0$ by Lemma~\ref{lem:Cameron-Martin}, which is a contradiction with $\P[\mathcal{N}_\infty(f,\ell)=k_0]=1$.
\end{proof}

\begin{proposition}\label{prop:no_infty}
	If $f$ satisfies assumptions \ref{a:1}, \ref{a:2} and \ref{a:3}, then for every $\ell\in\R$ one has
	$$\P[\mathcal{N}_\infty(f,\ell)=\infty]=0.$$
\end{proposition}

\begin{proof}
	Given a random field $g$ and $K\geq R>0$, we say that a point $x\in \R^d$ is an \emph{$(R,K)$-coarse trifurcation} for $\{g\geq\ell\}$ if the following happens: denoting by $\mathcal{C}_x$ the component of $x$ in $\{g\geq\ell\}$ and by $(\mathcal{C}_x^i)_{1\leq i\leq k(x)}$ the distinct components of $\mathcal{C}_x\setminus B_R(x)$, then $(\mathcal{C}_x^i)_{1\leq i\leq k(x)}$ are all connected to $x$ in $\mathcal{C}_x\cap B_K(x)$ and at least three of them are unbounded. We denote this event by $\mathrm{Trif}_x(g,\ell,R,K)$. Assume by contradiction that $\P[\mathcal{N}_\infty(\ell)=\infty]=1$. By proceeding as in the proof of Proposition~\ref{prop:no_finite>1}, we can find $r,M>0$ such that
	\begin{equation}\label{eq:quasi_trif>0}
		\P[\mathrm{Int}_3(f,\ell,r)\cap\{\inf_{x\in B_r} f(x)\geq-M\}]>0,
	\end{equation}
	where $\mathrm{Int}_3(f,\ell,r)$ denotes the event that at least three distinct unbounded components of $\{f\geq\ell\}$ intersect $B_r$. Furthermore, we can construct a function $h\in H$ (exactly as in the proof of Proposition~\ref{prop:no_finite>1}) such that $h\geq0$, $h\in C^2\cap W^{2,\infty}\cap W^{1,1}$ and $h(x)\geq M+\ell$ for all $x\in B_r$.
	By construction and the definition of global equivalence we have $\mathrm{Int}_3(f,\ell,r)\cap\{\inf_{x\in B_r} f(x)\geq -M\}\cap \textrm{GE}(f,\ell,h)\subset \cup_{R,K} \textrm{Trif}_0(f+h,\ell,R,K)$, thus $\P[\cup_{R,K}\textrm{Trif}_0(f+h,\ell,R,K)]>0$ by \eqref{eq:quasi_trif>0} and Lemma~\ref{lem:global_equiv}.
	In particular, there exists $K\geq R>0$ such that $\P[\textrm{Trif}_0(f+h,\ell,R,K)]>0$, and by Lemma~\ref{lem:Cameron-Martin}, 
	\begin{equation}\label{eq:trif>0}
		\P[\textrm{Trif}_0(f,\ell,R,K)]>0.
	\end{equation}
	We fix $R,K$ as above and consider $L\gg K$. Let $\mathcal{S}_L$ be the set of components of $\{f\geq\ell\}\cap \partial B_L$ and $\mathcal{X}_L\coloneqq \{x\in 4K\Z^d\cap B_{L-2K} :~ \mathrm{Trif}_x(f,\ell,R,K) \text{ happens}\}$ be the set of coarse trifurcations. Let $\mathcal{T}_L\coloneqq |\mathcal{X}_L|$ and recall that $\mathcal{N}_{\partial B_L}=|\mathcal{S}_L|$. One can prove that the following inequality holds deterministically
	\begin{equation}\label{eq:trif<boundary}
	\mathcal{T}_L\leq \max\{0,\mathcal{N}_{\partial B_L}-2\}\leq \mathcal{N}_{\partial B_L}.
	\end{equation}
	Before justifying \eqref{eq:trif<boundary}, let us conclude the proof. 
	On the one hand, by Lemma~\ref{lem:number_comp} we have $\E[\mathcal{N}_{\partial B_L}]\leq C_0L^{d-1}$.
	On the other hand, by \eqref{eq:trif>0} and stationarity we have $\E[\mathcal{T}_L]=|4K\Z^d\cap B_{L-2K}|\P[\mathrm{Trif}_0(f,\ell,R,K)]$ $\geq c(K)L^d$. Together with \eqref{eq:trif<boundary}, this gives a contradiction for $L$ large enough, thus concluding the proof.
	
	We now give a proof of \eqref{eq:trif<boundary}, which is close in spirit to that in \cite{BurKea89} -- see also \cite{AizDumSid15}  for an alternative proof. Let $x\in \mathcal{X}_L$ and assume without loss of generality that $\mathcal{C}_x^i$ is unbounded if and only if $i\leq k'(x)$, where $3\leq k'(x)\leq k(x)$. Now, for each $i\leq k'(x)$, let $P_x^i\subset \mathcal{S}_L$ be the (non-empty) set of boundary components belonging to $\mathcal{C}_x^i$. Notice that $\mathcal{P}_x\coloneqq \{P_x^1,\cdots,P_x^{k'(x)}\}$ is a sub-partition of $\mathcal{S}_L$, i.e.~$P_x^i\cap P_x^j=\emptyset$ for all $i\neq j$. Furthermore, for every $x,y\in \mathcal{X}_L$, $x\neq y$, one of the following happens:
	\begin{itemize}
		\item $\mathcal{P}_x$ and $\mathcal{P}_y$ are \emph{unrelated}, i.e.~one has $P_x^i\cap P_y^j=\emptyset$ for all $i\leq k'(x), j\leq k'(y)$. Indeed, this happens if $x$ and $y$ belong to distinct unbounded components.
		\item $\mathcal{P}_x$ and $\mathcal{P}_y$ are \emph{compatible}, i.e.~there exist $i\leq k'(x)$ and $i\leq k'(y)$ such that $\bigcup_{l\neq j} P_y^l \subset P_x^i$ and $\bigcup_{l\neq i} P_x^l \subset P_y^j$. Indeed, first notice that if $x$ and $y$ belong to the same unbounded components, then there exist $i\leq k'(x)$ and $j\leq k'(y)$ such that $y\in \mathcal{C}_x^i$ and $x\in \mathcal{C}_y^j$. By the definition of coarse trifurcation, one concludes that $\bigcup_{l\neq j} \mathcal{C}_y^l \subset \mathcal{C}_x^i$ and $\bigcup_{l\neq i} \mathcal{C}_x^l \subset \mathcal{C}_y^j$, and the compatibility of $\mathcal{P}_x$ and $\mathcal{P}_y$ follows.
	\end{itemize}
	The desired inequality \eqref{eq:trif<boundary} follows directly from the combinatorial lemma below.
	\begin{lemma}\label{lem:comp_part}
		Let $(\mathcal{P}_t)_{t\in T}$ be a family of sub-partition of a finite set $S$ with $|\mathcal{P}_t|\geq3$ for all $t\in T$. Assume that for every $t\neq s$, the sub-partitions $\mathcal{P}_t$ and $\mathcal{P}_s$ are either unrelated or compatible. Then $|T|\leq |S|-2$. 
	\end{lemma}
	A simple proof (by induction in $|T|$) of Lemma~\ref{lem:comp_part} can be found in \cite{BurKea89}. In fact, the original lemma from \cite{BurKea89} deals with full partitions (therefore never unrelated) of cardinality exactly $3$, but the proof applies readily to the context of Lemma~\ref{lem:comp_part} above.
	\end{proof}

We conclude by deducing Corollary~\ref{cor:uniq_nodal} from Theorem~\ref{thm:uniq}.

\begin{proof}[Proof of Corollary~\ref{cor:uniq_nodal}]
	Assume by contradiction that $\{f=\ell\}$ has two distinct unbounded components $A$ and $B$. First notice that by the smoothness assumption, $f$ is  $C^3$-smooth almost surely -- see e.g.~Sections A.3 and A.9 of \cite{NS16}. Also, by the non-degeneracy assumption, $f$ contains no critical point in $\{f=\ell\}$ almost surely, and therefore $A$ and $B$ are both closed, orientable and $C^1$-smooth hyper-surfaces. Hence by the Jordan--Brouwer separation theorem \cite{Lim88}, we have that $\R^n\setminus A = U \cup V$ with $U$ and $V$ two connected open subsets of $\R^n$ with boundary equal to $A$. Assume without loss of generality that $B\subset U$. Since the gradient of $f$ does not vanish on $A$ by non-degeneracy, $f-\ell$ changes its sign at $A$ and therefore there exists a neighborhood $U'$ of $A$  in $U$\footnote{i.e.~a set $U'\subset U$ such that for every $x\in A$ there exists an open set $O_x\subset \R^d$ containing $x$ such that $U\cap O_x\subset U'$} such that either $f<\ell$ or $f>\ell$ in $U'$. In the first case, the (unbounded) components of $A$ and $B$ in $\{f\geq \ell\}$ are separated by $U'$ and therefore disjoints, contradicting Theorem~\ref{thm:uniq}. In the second case, the (unbounded) components of $A$ and $B$ in $\{f\leq \ell\}$ are separated by $U'$ and therefore disjoints. Since $f$ is centered, $\{f\leq\ell\}$ has the same distribution as $\{f\geq-\ell\}$, and we obtain again a contradiction with Theorem~\ref{thm:uniq}.
\end{proof}

\paragraph{Acknowledgements.} I would like to thank Damien Gayet for helping with the topological argument in the proof of Corollary~\ref{cor:uniq_nodal} from Theorem~\ref{thm:uniq}. I would also like to thank Stephen Muirhead for providing references for the proofs of Lemmas~\ref{lem:number_comp} and \ref{lem:global_equiv}, as well as for pointing out that our assumptions are satisfied by fields with positive spectral density around the origin. Finally, I am grateful to the anonymous referees for their valuable comments on a previous version of this article. 
This project has received funding from the European Research Council (ERC) under the European Union’s Horizon 2020 research and innovation program (grant agreement No 851565).

\bibliographystyle{alpha}

\end{document}